\documentclass[reqno,centertags,12pt]{amsart}

\usepackage{amssymb,amsmath,amsthm}

\makeatletter \def\@strippedMR{} \def\@scanforMR#1#2#3\endscan{%
  \ifx#1M\ifx#2R\def\@strippedMR{#3}%
  \else\def\@strippedMR{#1#2#3}%
  \fi\fi} \renewcommand\MR[1]{\relax \ifhmode\unskip\spacefactor3000
  \space\fi \@scanforMR#1\endscan
  MR\MRhref{\@strippedMR}{\@strippedMR}} \makeatother

\newcommand{\R}{\mathbb{R}} \newcommand{\Z}{\mathbb{Z}}
 \newcommand{\C}{\mathbb{C}}

\theoremstyle{plain} \newtheorem{thm}{Theorem}[section]
\newtheorem{lem}[thm]{Lemma} 
\newtheorem{pro}[thm]{Proposition}

\theoremstyle{definition} \newtheorem{dfn}[thm]{Definition}
\theoremstyle{remark} \newtheorem{rem}{Remark}

 \newcommand{\lb}{\langle}
\newcommand{\rb}{\rangle} \newcommand{\ls}{\lesssim}
\newcommand{\gs}{\gtrsim}

\begin{document}

\title[The Zakharov-Kuznetsov equation]%
{The Fourier restriction norm method for the Zakharov-Kuznetsov equation}

\author[A.~Gr\"unrock]{Axel~Gr\"unrock}

\address{Heinrich-Heine-Universit\"at D\"usseldorf, Mathematisches Institut,
Universit\"atsstra{\ss}e 1, 40225 D\"usseldorf, Germany}
\email{gruenroc@math.uni-duesseldorf.de}

\author[S.~Herr]{Sebastian~Herr}

\address{Universit\"at Bielefeld, Fakult\"at f\"ur Mathematik,
  Postfach 10 01 31, 33501 Bielefeld, Germany}
\email{herr@math.uni-bielefeld.de}


\begin{abstract}
The Cauchy problem for the Zakharov-Kuznetsov equation is shown to be locally
well-posed in $H^s(\R^2)$ for all $s>\frac12$ by using the Fourier restriction norm method and bilinear refinements of Strichartz type inequalities.
\end{abstract}
\maketitle
\section{Introduction and main results}\label{sect:intro_main}
\noindent
We consider the initial value problem
\begin{equation}\label{eq:zk}
  \begin{split}
    \partial_t u +\partial_x^3u+\partial_x\partial_{yy}u&=\partial_x(u^2)\quad \text{in }(-T,T)\times \R^d,\\
    u(0,\cdot)&=\phi \in H^s(\R^d)
  \end{split}
\end{equation}
for the Zakharov-Kuznetsov equation, which is a higher dimensional
generalization of the Korteweg-de Vries equation. In three dimensions,
this equation has been derived by Zakharov and Kuznetsov
\cite[equation (6)]{ZK74} to describe unidirectional wave propagation
in a magnetized plasma. A derivation of the two-dimensional equation
considered here from the basic hydrodynamic equations was performed by
Laedke and Spatschek in \cite[Appendix B]{LS82}. A rigorous
justification of equation \eqref{eq:zk} from the Euler-Poisson system
for a uniformly magnetized plasma, valid in both considered space
dimensions, was given very recently by Lannes, Linares, and Saut in
\cite{LLSp}. Various aspects of the Zakharov-Kuznetsov equation and
its generalizations have attracted much attention in recent
years. Without attempting to be complete we refer to the papers
\cite{Sh90,F95,BL03,Pa04,F08,LP09,LPS10,ST10,LP11,RV12} and references
therein. In this paper we will focus on the case $d=2$.

\quad

Regular solutions preserve the $L^2(\R^d)$-norm.
Furthermore, there is an underlying Hamiltonian structure and
conservation of energy, cp. \cite{LS09} and references therein.

\quad

The three-dimensional version of the Cauchy problem \eqref{eq:zk} was
shown to be locally well-posed in $H^s(\R^3)$ for $s > \frac{9}{8} $
by Linares and Saut in \cite{LS09}, where they used the refined energy
method of Koch an Tzvetkov, see \cite{KT03}. This result has been
pushed down to $s>1$ recently by Ribaud and Vento \cite{RV11}, who
proved an essentially sharp maximal function estimate for the
linearized equation and combined this with the local smoothing effect.

\quad

Concerning the Cauchy problem on $\R^2$, global well-posedness in the
Sobolev space $H^1(\R^2)$ has been shown by Faminski{\u\i} in
\cite{F95}. Following the argument developed by Kenig, Ponce, and Vega
in \cite{KPV93} he proves the local smoothing effect, a maximal
function estimate as well as a Strichartz type inequality for the
linear equation to obtain local well-posedness by the contraction
mapping principle. The global result is then a consequence of the
conservation of energy. Linares and Pastor observed in
\cite[Theorem 1.6]{LP09} that Faminski{\u\i}'s proof can be optimized to
obtain local well-posedness in the larger data spaces $H^s(\R^2)$ with
$1 > s > \frac34$. To our knowledge this is the most advanced result
concerning the local problem up to date. We remark that
all of the above mentioned results rely on linear estimates to handle
nonlinear interactions of waves.
The purpose of this paper is  improve the well-posedness theory by using genuinely bilinear estimates.

\begin{thm}\label{thm:main}
  Let $s>\frac12$. The initial value problem \eqref{eq:zk} is locally
  well-posed in $H^s(\R^2)$.
\end{thm}

The scale invariant Sobolev regularity is $s_c = - 1$. We
expect that the regularity threshold can be improved further, but we
do not pursue this here.

\quad

The paper is organized as follows: We first
perform a linear transformation in the space variables $x$ and $y$,
see Subsection \ref{subsect:trafo} below. Then,
in Subsection \ref{subsect:fs}, we introduce the $X^{s,b}$- and restriction
spaces adapted to the linear part of the transformed equation and
recall the corresponding Strichartz-type estimates. Subsection
\ref{subsect:bil} is devoted to the proof of several bilinear space
time estimates for free solutions, which play a major role in our
analysis. In Section
\ref{sect:key} we prove our main bilinear estimate in Theorem
\ref{thm:key-est}.
\section{Preliminaries}\label{sect:pre}
\noindent
We start by fixing notation. Throughout this paper we denote the first spatial variable by $x$, its
dual Fourier variable by $\xi$, and the second spatial variable by $y$, and its dual Fourier variable
by $\eta$. As usual, $\tau$ is the dual variable of the time $t$. For $s \in \R$ $J^s$ and $I^s$ denote
the Bessel- and Riesz-potential operators of order $-s$ with respect to both spatial variables. The
corresponding one-dimensional operators will be called $J^s_x$ and $I^s_x$ respectively $J^s_y$ and $I^s_y$.
Moreover, we use the operator $\Lambda ^b := \mathcal{F}^{-1}\lb\tau -
\xi^3 -\eta^3\rb\mathcal{F}$, where for $a \in \R$ we set
$\lb a \rb := (1+a^2)^{\frac12}$. Projections onto dyadic intervals in Fourier space receive additional
subscripts, e.g. for $k \in \Z$ we define $P_{x,k}=
\mathcal{F}^{-1}\chi_{\{|\xi| \le 2^k\}} \mathcal{F}$, where $\chi$
denotes the (sharp) characteristic function.
$P_{x,\Delta k}=P_{x,k+1}-P_{x,k}$, $P_{x, \ge 1} = Id - P_{x,0}$, and similarly for the $y$- and $\eta$-variables.

\subsection{A linear transformation}\label{subsect:trafo}
We perform a linear change of variables (essentially a rotation) in order to symmetrize the
equation. A systematic study of such
transformations in connection with dispersive estimates for cubic
phase functions of two variables can be found in \cite{BKS03}. Let $x'=\mu x+\lambda y$ and $y'=\mu x -\lambda y$ and
$u'(x',y')=u(x,y)$. Then,
\begin{align*}
\partial_x u(x,y)=&\mu (\partial_{x'}+\partial_{y'}) v(x',y')\\
\partial_y u(x,y)=&\lambda (\partial_{x'}-\partial_{y'}) v(x',y')
\end{align*}
which implies
\begin{align*}
&(\partial_x^3+\partial_x\partial_y^2 )u(x,y)\\=&\mu^3
(\partial_{x'}+\partial_{y'})^3 v(x',y')+\mu \lambda^2
(\partial_{x'}+\partial_{y'})(\partial_{x'}-\partial_{y'})^2
v(x',y')\\
=&(\mu^3+\mu \lambda^2)(\partial_{x'}^3+\partial_{y'}^3) v(x',y')+(3\mu^3-\mu \lambda^2)(\partial_{x'}^2\partial_{y'}+\partial_{x'}\partial_{y'}^2) v(x',y').
\end{align*}
Choosing $\mu=4^{-\frac13}$ and $\lambda=\sqrt{3}4^{-\frac13}$ reduces
the above to
\begin{equation*}
(\partial_x^3+\partial_x\partial_y^2 )u(x,y)=(\partial_{x'}^3+\partial_{y'}^3) v(x',y')
\end{equation*}
which implies that we may consider the initial value problem
\begin{equation}\label{eq:zk-tr}
  \begin{split}
    \partial_t v +(\partial_x^3+\partial_y^3)v&=4^{-\frac13} (\partial_x+\partial_y)(v^2)\quad \text{in }(-T,T)\times \R^2,\\
    v(0,\cdot)&=\phi \in H^s(\R^2)
  \end{split}
\end{equation}
instead of \eqref{eq:zk} without changing the well-posedness theory.
We define the associated unitary group
$U(t):=e^{-t(\partial_x^3+\partial_y^3)}$.
\subsection{Function spaces and linear estimates}\label{subsect:fs}
In analogy with the KdV theory in \cite{B93,KPV96} we use
Bourgain's $X^{s,b}$ spaces. We refer the reader to the expositon in \cite[Section
2]{GTV97} for more details.
\begin{dfn}\label{dfn:spaces}
Let $s,b\in \R$. The space $X^{s,b}$ is defined as the space of all tempered distributions $u$ on
$\R\times \R^2$ such that $\widehat{u} \in L^2_{loc}(\R\times \R^2)$ and
\begin{equation}\label{eq:x-norm}
\|u\|_{s,b}:=\|\lb \tau -\xi^3-\eta^3\rb^b \lb \xi\rb^s
\widehat{u}(\tau,\xi,\eta) \|_{L^2_{\tau\xi\eta}}<+\infty.
\end{equation}
Furthermore, for $T>0$ we define the restriction space $X_T^{s,b}$ as
the space of all $u|_{(0,T)\times \R^2}$ for $u \in X^{s,b}$, with
norm
\begin{equation}\label{eq:x-norm-restr}
\|u\|_{s,b;T}:=\inf \{\|v\|_{s,b} : v \in X^{s,b}, v|_{(0,T)\times\R^2} =u\}.
\end{equation}
\end{dfn}
Finally, we define the set $X^{s,b}_{loc}$ of all $u$ satisfying $u|_{(0,T)\times \R^2}\in X_T^{s,b}$ for all $T>0$.

Let $\psi\in C_0^\infty(-2,2)$ be even, $0\leq \psi\leq 1$ and $\psi(t)=1$
for $|t| \leq 1$, and define $\psi_T(t):=\psi(t/T)$ for $T>0.$

The following result and its proof can be found in \cite[Lemma 2.1]{GTV97}.
\begin{lem}\label{lem:linear}
Let $s,b\in \R$. Then,
\begin{equation}\label{eq:linear-hom}
\|\psi U\phi\|_{s,b}\ls \|\phi\|_{H^s}.
\end{equation}
Also, for $-\frac12<b'\leq 0\leq b\leq b'+1$, $0<T\leq 1$,
\begin{equation}\label{eq:linear-inhom}
\|\psi_T \int_0^t U(t-s) f(s) ds\|_{s,b}\ls {T^{1-b+b'}}\|f\|_{s,b'}.
\end{equation}
\end{lem}

Next, let us recall two estimates of Strichartz type. \cite[Theorem 3.1
ii)]{KPV91} implies the estimate
\begin{equation}\label{eq:str1-free}
\|I_x^{\frac{1}{2p}} I_y^{\frac{1}{2p}}U\phi\|_{L^p_tL^q_{x,y}}\ls
\|\phi\|_{L^2_{x,y}}\, \qquad \text{ if } \frac{2}{p}+\frac{2}{q}=1, \, p>2.
\end{equation}
Sobolev embeddings imply
\begin{equation}\label{eq:str2-free}
\|U\phi\|_{L^p_tL^q_{x,y}}\ls
\|\phi\|_{L^2_{x,y}}\, \qquad \text{ if }
\frac{3}{p}+\frac{2}{q}=1, \, p> 3,
\end{equation}
which is a special case of estimate (A.11) from \cite{GS93}.

In the case $b>\frac12$ we can write any $u\in X^{0,b}$ as a superposition
of modulated free solutions, and \cite[Lemma 2.3]{GTV97} implies
\begin{lem}\label{lem:str}
Let $b>\frac12$.
\begin{equation}\label{eq:str1}
\|I_x^{\frac{1}{2p}} I_y^{\frac{1}{2p}}u\|_{L^p_tL^q_{x,y}}\ls
\|u\|_{0,b}\, \qquad \text{ if } \frac{2}{p}+\frac{2}{q}=1, \, p>2,
\end{equation}
and
\begin{equation}\label{eq:str2}
\|u\|_{L^p_tL^q_{x,y}}\ls
\|u\|_{0,b}\, \qquad \text{ if }
\frac{3}{p}+\frac{2}{q}=1, \, p> 3.
\end{equation}
\end{lem}
In particular, we obtain
\begin{equation}\label{eq:l4-str}
\|u\|_{L^4_{t,x,y}}\ls
\|u\|_{0,b} \qquad \text{ if } b>\frac{5}{12}
\end{equation}
by interpolation \eqref{eq:str2} for $p=5$ with the trivial bound $\|u\|_{L^2_{t,x,y}}=\|u\|_{0,0}$.
A further interpolation with the conservation of the $L^2$ - norm gives
\begin{equation}\label{eq:lpq-str}
\|u\|_{L^p_tL^q_{x,y}}\ls
\|u\|_{0,b} \quad \text{ if }\frac{2}{p}+\frac{2}{q}=1, \, p\ge 4 \quad \text{and} \quad b>\frac{2}{3p}+\frac{1}{q}.
\end{equation}

\subsection{Bilinear estimates for free solutions}\label{subsect:bil}
For a given measurable function $a:\R^4\to
\C$ of at most polynomial growth we define the bilinear operator $A$ with symbol $a$ via
\begin{equation*}
\widehat{A(f_1,f_2) }(\xi,\eta)=\int_{\xi=\xi_1+\xi_2\atop
  \eta=\eta_1+\eta_2} a(\xi_1,\xi_2,\eta_1,\eta_2)
\prod_{j=1}^2\widehat{f_j}(\xi_j,\eta_j)d\xi_1d\eta_1,
\end{equation*}
initially for $f_1,f_2 \in \mathcal{S}(\R^2)$.
Similar to \cite{G05} let $I_{x,-}^s$ be the bilinear
operator with symbol $|\xi_1-\xi_2|^s$,  $I_{x,+}^s$ be the bilinear
operator with symbol $|\xi_1+2\xi_2|^s$, $I_{y,-}^s$ be the bilinear
operator with symbol $|\eta_1-\eta_2|^s$ and $I_{y,+}^s$ be the bilinear
operator with symbol $|\eta_1+2\eta_2|^s$.
Convolution integrals as in the above definition will henceforth be abbreviated by
$\int_*$, e. g.

$$\int_* f(\xi_1)g(\xi_2)d\xi_1 := \int_{\xi_1+\xi_2 = \xi}f(\xi_1)g(\xi_2)d\xi_1 = \int f(\xi_1)g(\xi -\xi_1)d\xi_1$$

and similarly, if several variables appear.

\begin{pro}\label{pro:bil}
Let $b>\frac12$. Then,
\begin{align}
\label{eq:bil1}
\|I_x^{\frac12}I_{x,-}^{\frac12} (P_{y,k}u,v)\|_{L^2_{t,x,y}}\ls &
2^{\frac{k}{2}}\|u\|_{0,b}\|v\|_{0,b}\\
\label{eq:bil2}\|I_x^{\frac12}I_{x,-}^{\frac12}  (u,P_{y,k}v)\|_{L^2_{t,x,y}}\ls &
2^{\frac{k}{2}}\|u\|_{0,b}\|v\|_{0,b}\\
\label{eq:bil3}\|P_{y,k}I_x^{\frac12}I_{x,-}^{\frac12}  (u,v)\|_{L^2_{t,x,y}}\ls &
2^{\frac{k}{2}}\|u\|_{0,b}\|v\|_{0,b}
\end{align}
\end{pro}

\begin{proof}
By the  transfer principle (i.\ e.\  the multilinear
generalization of \cite[Lemma 2.3]{GTV97})
it suffices to prove the estimates for free solutions
 $u(t)=U(t)u_0$ and $v(t)=U(t)v_0$ with $\|u\|_{0,b}\|v\|_{0,b}$ replaced by
$\|u_0\|_{L^2_{x,y}} \|v_0\|_{L^2_{x,y}}$. In this case,
the Fourier transform of $I_x^{\frac12}I_{x,-}^{\frac12}  (u,v)$ in all three variables
is given as
\begin{align*}
 & \qquad \mathcal{F} I_x^{\frac12}I_{x,-}^{\frac12}  (u,v) (\xi,\eta,\tau) \\
= & \qquad c \int_* |\xi(\xi_1-\xi_2)|^{\frac12}\delta(\tau-\xi_1^3-\xi_2^3-\eta_1^3-\eta_2^3)
          \widehat{u_0}(\xi_1,\eta_1)\widehat{v_0}(\xi_2,\eta_2)d\xi_1 d \eta_1 \\
= & \qquad c \int_* |\xi(\xi_1^*-\xi_2^*)|^{-\frac12}(\widehat{u_0}(\xi_1^*,\eta_1)\widehat{v_0}(\xi_2^*,\eta_2)+
\widehat{u_0}(\xi_2^*,\eta_1)\widehat{v_0}(\xi_1^*,\eta_2)) d \eta_1.
\end{align*}
Here $\xi_1^*$ and $\xi_2^*$ are the solutions of $g(\xi_1^*)=0$, where $g(\xi_1)=\tau-\xi_1^3-(\xi-\xi_1)^3-\eta_1^3-\eta_2^3 $.
Observe that $\xi_1^*+\xi_2^*=\xi$, so by symmetry it suffices to consider only the first contribution to the above expression.
To see \eqref{eq:bil1} we assume $u=P_{y,k}u$ and use Cauchy-Schwarz to obtain the upper bound
$$2^{\frac{k}{2}}\left(\int_* |\xi(\xi_1^*-\xi_2^*)|^{-1}|\widehat{u_0}(\xi_1^*,\eta_1)\widehat{v_0}(\xi_2^*,\eta_2)|^2 d \eta_1\right)^{\frac12} .$$
Squaring and integrating with respect to $\tau$ leads to
\begin{align*}
 & \qquad \|\mathcal{F} I_x^{\frac12}I_{x,-}^{\frac12}  (u,v) (\xi,\eta,\cdot)\|^2_{L^2_{\tau}} \\
\ls & \qquad 2^k \int_* |\xi(\xi_1^*-\xi_2^*)|^{-1}|\widehat{u_0}(\xi_1^*,\eta_1)\widehat{v_0}(\xi_2^*,\eta_2)|^2 d \eta_1 d \tau \\
\ls & \qquad 2^k \int_* |\widehat{u_0}(\xi_1,\eta_1)\widehat{v_0}(\xi_2,\eta_2)|^2 d \eta_1 d \xi_1.
\end{align*}
 Now integration with respect to $\xi$ and $\eta$ gives (the square of) \eqref{eq:bil1}, the second estimate \eqref{eq:bil2} then obviously  holds true by symmetry.

\quad

Alternatively we can first take the $L^2_{\tau}$ - norm and apply Minkowski's integral inequality to obtain

\begin{align*}
 & \qquad \|\mathcal{F} I_x^{\frac12}I_{x,-}^{\frac12}  (u,v) (\xi,\eta,\cdot)\|_{L^2_{\tau}} \\
\ls & \qquad \int_* \||\xi(\xi_1^*-\xi_2^*)|^{-\frac12}(\widehat{u_0}(\xi_1^*,\eta_1)\widehat{v_0}(\xi_2^*,\eta_2)\|_{L^2_{\tau}}
       d \eta_1 =: I(\xi,\eta).
\end{align*}

Now the square of the norm inside the integral equals

$$\int_* |\widehat{u_0}(\xi_1,\eta_1)\widehat{v_0}(\xi_2,\eta_2)|^2 d \xi_1,$$ 

so that by a second application of Minkowski's inequality

$$\|I(\cdot,\eta)\|_{L^2_{\xi}} \ls \int_* \|\widehat{u_0}(\cdot,\eta_1)\|_{L^2_{\xi}}\|\widehat{v_0}(\cdot,\eta_2)\|_{L^2_{\xi}}
          d \eta_1 \ls \|u_0\|_{L^2_{xy}}\|v_0\|_{L^2_{xy}},$$

which gives a bound independent of $\eta$. Finally we use $\|P_{y,k}F\|_{L^2_{\eta}} \ls 2^{\frac{k}{2}} \|F\|_{L^{\infty}_{\eta}}$
to obtain \eqref{eq:bil3}.
\end{proof}

\qquad

\begin{rem}\label{rem:bil}
The above proposition has several useful consequences:
\begin{enumerate}
\item As the proof shows, we may replace the dyadic intervals
symmetric around zero by intervals $I$ of arbitrary position and length
$|I|$, if we change the factor $2^{\frac{k}{2}}$ on the right into
$|I|^{\frac12}$. In case of \eqref{eq:bil1} the position of $|I|$ may even
depend on $\eta $.
\item Summing up the dyadic pieces in \eqref{eq:bil1} - \eqref{eq:bil3} and using multilinear interpolation we obtain for
$s_0, s_1, s_2 \geq 0$ with $s_0 + s_1+s_2>\frac12$ and $b>\frac12$ the inequality
\begin{equation}\label{eq:bil4}
\|J_y^{-s_0}I_x^{\frac12}I_{x,-}^{\frac12} (u,v)\|_{L^2_{t,x,y}}\ls 
\|J_y^{s_1}u\|_{0,b}\|J_y^{s_2}v\|_{0,b}.
\end{equation}
\item By symmetry in $x$ and $y$ we see that all the inequalities
\eqref{eq:bil1} - \eqref{eq:bil4} are equally valid with $x$ and $y$
interchanged.
\end{enumerate}
\end{rem}

\section{The key estimate}\label{sect:key}
Now we are prepared to prove the key estimate for the proof of Theorem \ref{thm:main}.
\begin{thm}\label{thm:key-est}
Let $s>\frac12$. Then for any $b'\le -\frac13$ and for any
$b>\frac12$ the estimate
\begin{equation}\label{eq:key-est}
\|(\partial_x+\partial_y)(u_1u_2)\|_{s,b'}\ls \|u_1\|_{s,b}\|u_2\|_{s,b}
\end{equation}
holds true for all $u_1,u_2\in X^{s,b}$.
\end{thm}
\begin{proof}
Throughout this proof let $\ast$ denote the convolution constraint
\[
(\tau,\xi,\eta)=(\tau_1,\xi_1,\eta_1)+(\tau_2,\xi_2,\eta_2)
\]
Under the above constraint it is obvious that
\begin{equation*}
\lb (\xi,\eta)\rb^s\ls \lb (\xi_1,\eta_1)\rb^s+\lb (\xi_2,\eta_2)\rb^s
\ls \lb (\xi_1,\eta_1)\rb^s\lb (\xi_2,\eta_2)\rb^s
\qquad (s>0)
\end{equation*}
holds true, which implies that it suffices to prove the claim in the
case $\frac12 < s \le \frac34$.
Let $\sigma_0:=
\tau-\xi^3-\eta^3$, $\sigma_j:=
\tau_j-\xi_j^3-\eta_j^3$, and
$f_j(\tau_j,\xi_j,\eta_j):=|\widehat{u_j}(\tau_j,\xi_j,\eta_j)|\lb\sigma_j\rb^b$
for $j=1,2$. The claim is equivalent to the following weighted
$L^2$ convolution estimate:
\begin{equation}\label{eq:key-est-l2}
\big\|M(f_1,f_2)\big\|_{L^2}\ls \prod_{j=1}^2\|f_j\|_{L^2}
\end{equation}
where
\[
M(f_1,f_2)(\tau,\xi,\eta):=\frac{(\xi+\eta)\lb (\xi,\eta)\rb^s}{\lb \sigma_0\rb^{-b'}} \int_{\ast} \prod_{j=1}^2\frac{f_j(\tau_j,\xi_j,\eta_j)}{\lb\sigma_j\rb^{b}\lb (\xi_j,\eta_j)\rb^s} d\tau_1d\xi_1d\eta_1.
\]
To show \eqref{eq:key-est-l2} we may assume by symmetry that $|\eta| \le |\xi|$. Then we split the domain of integration into three regions, which induces the following decomposition:
\[
\big\|M(f_1,f_2)\big\|_{L^2}=R_1+R_2+R_3+R_4.
\]

{\it Contribution $R_1$:} This corresponds to the region $|\xi| \ls |\xi_1-\xi_2|$, so that $|\xi| \ls |\xi|^{\frac12}|\xi_1-\xi_2|^{\frac12}$.
Assuming in addition that $|(\xi_1, \eta_1)| \ge |(\xi_2, \eta_2)|$, which can be done without loss of generality, we obtain the bound
$$R_1\ls \|I_x^{\frac12}I_{x,-}^{\frac12}(J^su_1,u_2)\|_{L^2_{t,x,y}}\ls \|u_1\|_{s,b}\|u_2\|_{s,b},$$
where we have used \eqref{eq:bil4} with $s_0=s_1=0$ and $s_2 = s > \frac12$.
 
{\it Contribution $R_2$:}  This corresponds to the region where $|\xi_1-\xi_2| \ll |\xi|$  and $|\eta_1| \gs |\xi|$. Here, we have $|\xi| \sim |\xi_1| \sim |\xi_2|$ and obtain
$$R_2\ls \|(I_x^{\frac{1-s}{2}}I_y^{\frac{1-s}{2}}J^su_1)(I_x^su_2)\|_{L^2_{t,x,y}}
+ \|(I_x^{\frac12}I_y^{\frac12}u_1)(J^su_2)\|_{L^2_{t,x,y}}=: R_{2.1} +  R_{2.2},$$
where
$$ R_{2.1} \le \|I_x^{\frac{1-s}{2}}I_y^{\frac{1-s}{2}}J^su_1\|_{L^p_tL^q_{x,y}}
\|J^su_2\|_{L^{\tilde{p}}_tL^{\tilde{q}}_{x,y}},$$
whenever $\frac{1}{p}+\frac{1}{q}=\frac{1}{p}+\frac{1}{\tilde{p}}=\frac{1}{q}+\frac{1}{\tilde{q}}=\frac12$ and $p>2$.
Choosing $\frac{1}{p}=1-s$ we can apply the Strichartz-type estimate \eqref{eq:str1} to bound the first factor by
$\|u_1\|_{s,b}$, while for the second we use the estimate \eqref{eq:lpq-str}, so that we arrive at
\[ R_{2.1} \ls \|u_1\|_{s,b}\|u_2\|_{s,b}.\] The contribution $R_{2.2}$ can be dealt with in exactly the same manner.

{\it Contribution $R_3$:} We consider the region where $|\xi_1-\xi_2| \ll |\xi|$  and $|\eta_2| \gs |\xi|$.
Here, the same argument as for $R_2$ applies (with $u_1$ and $u_2$ interchanged).

{\it Contribution $R_4$:} Here, we assume $|\xi_1-\xi_2| \ll |\xi|\sim |\xi_1| \sim |\xi_2|$ and $|\eta_1| \ll |\xi|$ and $|\eta_2| \ll |\xi|$, thus completing the case by case discussion. We observe that under the convolution constraint $*$ the
\emph{resonance} identity
\begin{equation}
\sigma_0-\sigma_1-\sigma_2=3(\xi\xi_1\xi_2+\eta\eta_1\eta_2) \label{eq:res}
\end{equation}
holds true (This is similar to the low regularity analysis of the KdV equation, where the analogous identity has been observed in \cite[formula 7.46]{B93}, see also \cite{KPV96}. Note that this similarity is due the transformation performed in Subsection \ref{subsect:trafo}.). In region $R_4$ this identity implies the inequality
\begin{equation}\label{eq:resineq}
 \lb \sigma_0 \rb + \lb \sigma_1 \rb + \lb \sigma_2 \rb \gs |\xi|^3,
\end{equation}
which naturally leads to the following further division $R_4=R_{4.0}+R_{4.1}+R_{4.2}$.

{\it Contribution $R_{4.0}$:} This corresponds to the subregion where $\lb \sigma_0 \rb \gs \lb \sigma_{1} \rb,\lb \sigma_{2} \rb$. Using \eqref{eq:resineq} we estimate
of \eqref{eq:key-est-l2} by
$$R_{4.0}\ls \|J^s(u_1u_2)\|_{L^2_{t,x,y}} \ls \|J^su_1\|_{L^4_{t,x,y}}\|J^su_2\|_{L^4_{t,x,y}} \ls \|u_1\|_{s,b}\|u_2\|_{s,b},$$
where in the last step we have used the estimate \eqref{eq:l4-str}.

{\it Contribution $R_{4.1}$:} Here, we consider the subregion where $\lb \sigma_1 \rb \gs \lb \sigma_{0} \rb,\lb \sigma_{2} \rb$. We recall the operator
 $\widehat{\Lambda^bu_1}(\tau_1,\xi_1,\eta_1)=\lb \sigma_1 \rb^b \widehat{u_1}(\tau_1,\xi_1,\eta_1)$. Using \eqref{eq:resineq} as well as $\lb \sigma_1 \rb \gs \lb \sigma_0 \rb$
we obtain the upper bound
\begin{align*}
R_{4.1}\ls& \|J^s((\Lambda^bu_1)(u_2))\|_{0,-b} \ls  \|(J^s\Lambda^bu_1)(J^su_2)\|_{L^{\frac43}_{t,x,y}} \\\ls& \|J^s\Lambda^bu_1\|_{L^2_{t,x,y}}\|J^su_2\|_{L^4_{t,x,y}},
\end{align*}
where first the dual version of \eqref{eq:l4-str} and then this estimate itself were applied.

{\it Contribution $R_{4.2}$:}  This corresponds to the subregion where $\lb \sigma_2 \rb \gs \lb \sigma_{0} \rb, \lb \sigma_{1} \rb$. This can be treated in precisely the same manner as $R_{4.1}$.
\end{proof}

For the sake of completeness, we conclude this paper with a sketch of the proof of  Theorem \ref{thm:main} based on Theorem \ref{thm:key-est}. The ideas are well-known, see e.g. \cite{B93,KPV96,GTV97}. For $\phi \in H^s(\R^2)$ we solve the integral equation associated to \eqref{eq:zk-tr}
\[
u(t)=U(t)\phi+ \mathcal{I}(u)(t), \quad \mathcal{I}(u)(t):=4^{-\frac13}\int_0^tU(t-s) (\partial_{x}+\partial_{y})u^2(s) ds
\]
in $X^{s,b}_T$ by means of the contraction mapping principle. Indeed, from Lemma \ref{lem:linear} and Theorem \ref{thm:key-est} it follows that
\begin{align*}
\|U(t)\phi +\mathcal{I}(u)\|_{s,b;T} &\ls \|\phi\|_{H^s}+T^\delta\|(\partial_{x}+\partial_{y})u^2\|_{s,b';T}\\
&\ls \|\phi\|_{H^s}+T^\delta \|u\|^2_{s,b;T}
\end{align*}
for some $b>\frac12$, $b'<-\frac13$ and $\delta>0$, and similarly
\[\|\mathcal{I}(u)-\mathcal{I}(v)\|_{s,b;T}
\ls T^\delta( \|u\|_{s,b;T}+\|v\|_{s,b;T})\|u-v\|_{s,b;T}.
\]
This implies existence of a fixed point $u \in X^{s,b}_T\subseteq C([-T,T],H^s(\R^2))$ for suitably chosen $T>0$ (depending on $\|\phi\|_{H^s}$). Based on these estimates one can also prove uniqueness of $u \in X^{s,b}_T$ and continuous dependence on the initial data. 
\bibliographystyle{plain} \bibliography{zak-kuz-2d.bib}
\end{document}